\documentclass[11pt]{amsart}
\usepackage{amsmath,amsthm,amssymb,amsfonts,mathrsfs,todonotes,tikz-cd,bm}

\usepackage[shortlabels]{enumitem}
\usepackage[all]{xy}
\usepackage{amsmath}

\usepackage[margin=1.3in]{geometry}
\usepackage{upgreek}
\usepackage{hyperref}

\hypersetup{
    colorlinks=true,
    linkcolor=blue,
    filecolor=magenta,      
    urlcolor=cyan,
    bookmarks=true,
    pdfpagemode=FullScreen,
}

\newtheorem{proposition}{Proposition}[subsection]

\newcommand{\diag}{\operatorname{diag}}
\newcommand{\n}{\mathfrak{n}}
\renewcommand{\H}{\operatorname{H}}
\newcommand{\Lie}{\operatorname{Lie}}
\newcommand{\Stab}{\operatorname{Stab}}

\newcommand{\supder}{\operatorname{SupDer}}
\newcommand{\innsupder}{\operatorname{InnSupDer}}

\newcommand\im{\operatorname{Im}}
\newcommand{\coker}{\operatorname{Coker}}
\renewcommand\ker{\operatorname{ker}}
\renewcommand\H{\operatorname{H}}

\newcommand\Hom{\operatorname{Hom}}

\newcommand{\overbar}[1]{\mkern 1.5mu\overline{\mkern-1.5mu#1\mkern-1.5mu}\mkern 1.5mu}

\newtheorem{theorem}{Theorem}[subsection]

\usepackage[utf8]{inputenc}

\title[Cohomology Groups for BBW Parabolics for Lie Superalgebras]{On First and Second Cohomology Groups for BBW Parabolics for classical Lie Superalgebras}

\author{David M. Galban}
\address
{Department of Mathematics\\ University of Georgia \\Athens\\ GA~30602, USA}

\email{david.galban25@uga.edu}
\thanks{The author was partially supported by NSF (RTG) grant DMS-1344994}
\begin{document}

\maketitle

\begin{abstract}Let ${\mathfrak g}$ be a classical simple Lie superalgebra.  In this paper, the author studies the cohomology groups for the subalgebra $\mathfrak{n}^{+}$ relative to the BBW parabolic subalgebras constructed by D. Grantcharov, N. Grantcharov, Nakano and Wu. These classical Lie superalgebras have a triangular decomposition ${\mathfrak g}={\mathfrak n}^{-}\oplus {\mathfrak f} 
\oplus {\mathfrak n}^{+}$ where $\mathfrak f$ is a detecting subalgebra as introduced by Boe, Kujawa and Nakano. It is shown that there exists a Hochschild-Serre spectral sequence that collapses for all infinite families of classical simple Lie superalgebras.  This enables the author to explicitly compute the first and second cohomologies for ${\mathfrak n}^{+}$. The paper concludes with tables listing the weight space decompositions and dimension formulas for these cohomology groups.
\end{abstract}

\section{Introduction}

\subsection{\ } For $\mathfrak{g}$ a semisimple Lie algebra over $\mathbb{C}$, $J$ a subset of simple roots and ${\mathfrak p}_{J}={\mathfrak l}_{J}\oplus {\mathfrak u}_{J}$ the corresponding parabolic subalgebra, a famous theorem of Kostant demonstrates that 
$$\operatorname{H}^k(\mathfrak{u}_J, L(\mu))= \bigoplus_{w \in W^{J}, \, l(w)=k} L_J(w \cdot \mu),$$
where $L_J(w \cdot \mu)$ is an irreducible finite-dimensional module corresponding to the Levi factor ${\mathfrak l}_{J}$ for $J$ \cite{UGA}. 
Kostant's theorem is piece of a larger picture where in the (parabolic) Category ${\mathcal O}_{J}$ one has the isomorphism: 
\begin{equation} \label{e:Kostant-KL} 
\text{Ext}^{n}_{{\mathcal O}_{J}}(Z_{J}(\lambda),L(\mu))\cong 
\text{Hom}_{{\mathfrak l}_{J}}(L_{J}(\lambda),\text{H}^{n}({\mathfrak u}_{J},L(\mu))),
\end{equation} 
where $Z_{J}(\lambda)$ is a (parabolic) Verma module arising from inducing a finite-dimensional ${\mathfrak l}_{J}$-module $L_{J}(\lambda)$ and $L(\lambda)$ is an irreducible representation in ${\mathcal O}_{J}$. 
It is a deep theorem that these extension groups in (\ref{e:Kostant-KL}) can be computed via Kazhdan-Lusztig polynomials \cite{kumar}.

\subsection{\ } In the case when ${\mathfrak g}$ is a classical simple Lie superalgebra one would like to have a Kazhdan-Lusztig theory and a Kostant-type theorem in the context of a Category ${\mathcal O}$ theory. D. Grantcharvov, N. Grantcharov, Nakano and Wu \cite{GGNW} introduced the notion of a BBW parabolic subalgebra, ${\mathfrak b}$, that contains the detecting subalgebra, ${\mathfrak f}$,  earlier introduced by Boe, Kujawa and Nakano \cite{BKN}. One can view that algebra ${\mathfrak f}$ like a Levi subalgebra and ${\mathfrak b}$ as a parabolic containing ${\mathfrak f}$. There exists a natural triangular decomposition of ${\mathfrak g}={\mathfrak n}^{-}\oplus {\mathfrak f}\oplus {\mathfrak n}^{+}$ where ${\mathfrak b}={\mathfrak f}\oplus {\mathfrak n}^{+}$ where the Lie superalgebras ${\mathfrak n}^{\pm}$ are nilpotent subalgebras. 

Recently, Lai, Nakano and Wilbert \cite{LNW}  have constructed a Category ${\mathcal O}_{\mathfrak f}$ via this triangular decomposition and have proved an analog to (\ref{e:Kostant-KL}). 
Other efforts have been made in understanding a Category ${\mathcal O}$ for Lie superalgebras on a case-to-case basis; however, prior to \cite{LNW} there has not been a unified treatment.  A fundamental question is to compute $\text{H}^{n}({\mathfrak n}^{+},L(\lambda))$, where $L(\lambda)$ is a finite-dimensional ${\mathfrak g}$-module and to determine if there is a 
Kostant-type theorem in the ${\mathcal O}_{\mathfrak f}$.  This paper aims to provide the first calculation in this direction. 

\subsection{Outline}
The paper is organized as follows.  In Section 2 we review the definitions of Lie superalgebras,  Lie superalgebra cohomology, and detecting and nilpotent subalgebras.  In Section 3, a Hochschild-Serre spectral sequence is defined for each of the infinite families of classical Lie superalgebras and it is shown that in each case, it collapses.

In Section 4, the notion of a superderivation is defined and it is shown how first cohomology for arbitrary modules can be expressed as a quotient of the set of superderivations. 
We then provide a formula for $\operatorname{H}^1(\mathfrak{n},\mathbb{C})$ and compute its dimension.

In Section 5, we first interpret second cohomology as giving the set of classes of central extensions for superalgebras.  Expressions for $\operatorname{H}^2(\mathfrak{n},\mathbb{C})$ for each classical Lie superalgebra in terms of their weight spaces are found, as well as formulas for their dimension.  Finally, in Section 6 we summarize the weight spaces and dimensions of both the $\operatorname{H}^1$ and $\operatorname{H}^2$ cohomologies in a series of tables.

\subsection{Acknowledgements} This paper is part of the author's Ph.D dissertation at the University of Georgia. He acknowledges his Ph.D advisor, Daniel K. Nakano, for his 
guidance throughout the project. He also thanks Shun-Jen Cheng for his insights about the exceptional families of simple Lie superalgebras.  

\section{Preliminaries}

\subsection{Notation}

Throughout this paper, all vector spaces, unless otherwise noted, will be over $\mathbb{C}$.
A superspace is a vector space $V=V_{\bar{0}} \oplus V_{\bar{1}}$ with a $\mathbb{Z}_2$-grading.  An element $v \in V_{\bar{0}}$ is referred to as even, and an element in $V_{\bar{1}}$ as odd.  Such an element in either $V_{\bar{0}}$ or $V_{\bar{1}}$ is referred to as homogeneous.  If $v$ is homogeneous, we define the degree $|v|$ of $v$ as the element $i \in \mathbb{Z}_2$ such that $v \in V_i$.

A Lie superalgebra is a superspace $\mathfrak{g} = \mathfrak{g}_{\bar{0}} \oplus \mathfrak{g}_{\bar{1}}$ equipped with a bilinear multiplication $[\cdot, \cdot]$ satisfying the following properties:
\begin{enumerate}
    \item $[\mathfrak{g}_i,\mathfrak{g}_j] \subseteq \mathfrak{g}_{i+j}$
    \item $[a,b]= -(-1)^{|a|\cdot |b|}[b,a]$
    \item $[a,[b,c]]=[[a,b],c]+(-1)^{|a| \cdot |b|}[b,[a,c]],$
\end{enumerate}
where properties 2 and 3 hold for homogeneous elements, and the multiplication is extended to all of $\mathfrak{g}$ linearly \cite[Definition 1.3]{CW}.
A $\mathfrak{g}$-module $M$ is a superspace equipped with an action by $\mathfrak{g}$ that is compatible with the $\mathbb{Z}_2$ grading.

The notion of a universal enveloping algebra generalizes to the superalgebra case as well.  Given a superalgebra $\mathfrak{g}$ let $T(\mathfrak{g})$ denote the tensor algebra on $\mathfrak{g}$.  Let $I$ denote the ideal generated by elements of the form
$$x \otimes y - (-1)^{|x||y|} y \otimes x - [xy]$$
Let $U(\mathfrak{g}) = T(\mathfrak{g})/I$ and let $i$ be the canonical embedding of $\mathfrak{g}$ into $U(\mathfrak{g})$.  Then $U(\mathfrak{g})$ satisfies the universal property that if $j: \mathfrak{g} \to M$ is any linear map satisfying

$$j([xy]) = j(x)j(y) - (-1)^{|x||y|}j(y)j(x)$$
then there is a unique homomorphism $\phi:U(\mathfrak{g}) \to M$ such that $\phi \circ i = j$.
We let $I\mathfrak{g}$ denote the augmentation ideal of $U(\mathfrak{g})$.

\subsection{Lie superalgebra cohomology}

We define the Lie superalgebra cohomology of $\mathfrak{g}$ with coefficients in a module $M$ as follows.  Consider the Koszul complex  whose cochain groups are given as
$$C^n(\mathfrak{g},M) = \Hom(\Lambda_s^n(\mathfrak{g}),M),$$
where $\Lambda_s^n(\mathfrak{g})$ denotes the superexterior algebra
$$\Lambda_s^n(\mathfrak{g}) := \bigoplus_{i+j=n}\Lambda^i(\mathfrak{g}_{\bar{0}}) \otimes S^j(\mathfrak{g}_{\bar{1}}).$$
The differential maps $d^n:C^n(\mathfrak{g},M) \rightarrow C^{n+1}(\mathfrak{g},M)$, for homogeneous $f$, are given by the formula
\begin{equation} \label{e:differential}
\begin{aligned}
df(\omega_0 \wedge \cdots \wedge \omega_{n}){} = & \sum_{i=0}^n (-1)^{\tau_i}\omega_i \cdot f(\omega_0 \wedge \cdots \wedge \widehat{\omega_i} \wedge \cdots \wedge \omega_n) \\ & + \sum_{i<j} (-1)^{\sigma_{i,j}}f([\omega_i,\omega_j] \wedge \omega_0 \wedge \cdots \widehat{\omega_i} \cdots \widehat{\omega_j} \cdots \wedge \omega_n),
\end{aligned}
\end{equation}

where
$$\tau_i = i +|\omega_i|(|\omega_0|+ \cdots + |\omega_{i-1}| + |f|),$$
and
$$\sigma_{i,j}= i+j+|\omega_i||\omega_j| +|\omega_i|(|\omega_0|+ \cdots + |\omega_{i-1}|)+|\omega_j|(|\omega_0|+ \cdots + |\omega_{j-1}|),$$
and which is then extended linearly to all of $C^n(\mathfrak{g},M)$.  It follows that $d^n \circ d^{n-1} = 0$, and so we define the $n$th cohomology group as
$$\operatorname{H}^n(\mathfrak{g},M) = \ker{d^n}/\im{d^{n-1}}.$$
Letting $\mathbb{C}$ denote the $\mathfrak{g}$-module concentrated in the even component of dimension 1 on which $\mathfrak{g}$ acts trivially, we define the cohomology of $\mathfrak{g}$ as $\operatorname{H}^n(\mathfrak{g}, \mathbb{C}).$

\subsection{Detecting and nilpotent subalgebras}

We now define the notion of a detecting subalgebra, essentially an analog of the Cartan subalgebra in the classical case, following D. Grantcharov, N. Grantcharov, Nakano, and Wu \cite{GGNW}.  We say that a Lie superalgebra $\mathfrak{g}$ is classical if  there is a connected reductive algebraic group $G_{\bar{0}}$ such that $\Lie(G_{\bar{0}}) = \mathfrak{g}_{\bar{0}}$ and if the action of $G_{\bar{0}}$ on $\mathfrak{g}_{\bar{1}}$ differentiates to the adjoint action.

If $\mathfrak{g}$ is a classical Lie superalgebra, $\mathfrak{g}_{\bar{1}}$ admits a stable action by $G_{\bar{0}}$.  Following the construction in \cite[Section 8.9]{BKN}, fix a generic element $x_0 \in \mathfrak{g}_{\bar{1}}$ and set $\displaystyle{H=\Stab_{G_{\bar{0}}}x_0}$. We define $\mathfrak{f}_{\bar{1}} = \mathfrak{g}_{\bar{1}}^H$ and
$\mathfrak{f}_{\bar{0}} = [\mathfrak{f}_{\bar{1}},\mathfrak{f}_{\bar{1}}]$
and let $\mathfrak{f} = \mathfrak{f}_{\bar{0}} \oplus \mathfrak{f}_{\bar{1}}$ be the \textit{detecting subalgebra}.

Moreover, as per \cite[Section 8]{BKN}, we can make the odd roots corresponding to $\mathfrak{f}$ explicit and thus also those corresponding to $\mathfrak{f}_{\bar{1}}$ and $\mathfrak{f}$ itself.  By convention, let $r$ denote the minimum of $m$ and $n$. Let $\Omega$ denote the set of odd roots of $\mathfrak{f}$. Then
$$\mathfrak{f}_{\bar{1}} = \{\sum_{\alpha \in \Omega} (u_\alpha x_\alpha + v_\alpha x_{-\alpha}) \mid  u_\alpha, v_\alpha \in \mathbb{C}\}.$$
$\mathfrak{f}_{\bar{0}}$ can then be obtained by taking brackets.

Let $\epsilon_i$ and $\delta_j$ be linear functionals on diagonal matrices
$$a = \diag(a_1, \cdots, a_{n+m})$$
which satisfy
$$\epsilon_i(a) = a_i$$
and
$$\delta_j(a) = a_{m+j}.$$
Then for each of the classical simple Lie superalgebras, we have the following values for $\Omega$.
\begin{center}
 \begin{tabular}{||c c||} 
 \hline
 $\mathfrak{g}$ & $\Omega$ \\ [0.5ex] 
 \hline\hline
 $\mathfrak{gl}(m | n)$ & $\{\epsilon_i-\delta_i  \:|\: 1 \leq i \leq r\}$  \\ 
 \hline
 $\mathfrak{sl}(m|n)$ & $\{\epsilon_i-\delta_i \:|\: 1 \leq i \leq r\}$  \\
 \hline
 $\mathfrak{psl}(n|n)$ & $\{\epsilon_i-\delta_i \:|\: 1 \leq i \leq n\}$  \\
 \hline
 $\mathfrak{osp}(2m+1|2n)$ & $\{\epsilon_i-\delta_i \:|\: 1 \leq i \leq r\}$  \\
 \hline
 $\mathfrak{osp}(2m|2n)$ & $\{\epsilon_i-\delta_i \:|\:1 \leq i \leq r\}$  \\ [1ex] 
 \hline
 $D(2,1;\,\alpha)$ & $\{\epsilon_1+\epsilon_2+\epsilon_3\}$ \\
 \hline
 $G(3)$ & $\{\epsilon_1+\delta\}$ \\
 \hline
 $F(4)$ & $\{\frac{\epsilon_1+\epsilon_2+\epsilon_3+\epsilon_4}{2}\}$ \\
 \hline
\end{tabular}
\end{center}
In the case of $\mathfrak{q}(n)$ we let $\mathfrak{f}_{\bar{1}}$ be the collection of all matrices whose odd part is diagonal.

Looking at the adjoint action of the maximal torus in $\mathfrak{f}_{\bar{0}}$ on $\mathfrak{g}$ produces a root-space decomposition of $\mathfrak{g}$, and letting $\mathfrak{n}$ denote the space of positive roots and $\mathfrak{n}^-$ the space of negative ones, we obtain a triangular decomposition $\mathfrak{g} = \mathfrak{n}^- \oplus \mathfrak{f} \oplus \mathfrak{n}$.
We also provide a table listing the collection of root spaces corresponding to each of the $\mathfrak{n}^-$ for the classical Lie superalgebras.

\begin{center}
 \begin{tabular}{||c c||} 
 \hline
 $\mathfrak{g}$ & $\Phi_{\bar{1}}^-$ \\ [0.5ex] 
 \hline\hline
 $\mathfrak{gl}(m | n)$ & $\{\epsilon_i-\delta_j, -\delta_i+\epsilon_j \:|\: i<j\}$  \\ 
 \hline
 $\mathfrak{sl}(m|n)$ & $\{\epsilon_i-\delta_j, -\delta_i+\epsilon_j \:|\: i<j\}$  \\
 \hline
 $\mathfrak{osp}(2m+1|2n)$ & $\{-\epsilon_i+\delta_j,\, -\delta_i+\epsilon_j, \, -\epsilon_k-\delta_l, \, -\delta_t \:|\: i < j\}$  \\
 \hline
 $\mathfrak{osp}(2m|2n)$ & $\{\epsilon_i-\delta_j, -\delta_i+\epsilon_j, -\epsilon_k-\delta_l \:|\: i<j\}$  \\ 
 \hline
 $\mathfrak{q}(n)$ & $\{\epsilon_i+\epsilon_j \:|\: i<j \}$\\
 \hline
 $D(2,1;\alpha)$ & $\{(-\epsilon, -\epsilon, -\epsilon)$, $(-\epsilon, -\epsilon, \epsilon)$, $(\epsilon, -\epsilon, -\epsilon)\}$ \\
 \hline
 G(3) & $\{(-\omega_1+\omega_2, -\epsilon)$,\newline $(2\omega_1-\omega_2, -\epsilon)$,\newline $(0, -\epsilon)$,\newline $(\omega_1-\omega_2, -\epsilon)$, \newline$(-2\omega_1+\omega_2, -\epsilon)$, \newline$(-\omega_1, -\epsilon)\}$ \\
 \hline
 F(4) & $\{(\omega_2-\omega_3, -\epsilon)$,\newline $(\omega_1-\omega_2+\omega_3, -\epsilon)$,\newline $(\omega_1-\omega_3, -\epsilon)$, \newline$(-\omega_2+\omega_3, -\epsilon)$\\
  & $(-\omega_1+\omega_2-\omega_3, -\epsilon)$,\newline $(-\omega_1+\omega_3, -\epsilon)$,\newline $(-\omega_3, -\epsilon)\}$ \\
 \hline
\end{tabular}
\end{center}

\section{The Hochschild-Serre Spectral Sequence}
As in the case of classical Lie algebra cohomology, letting $\mathfrak{h}$ denote an ideal of $\mathfrak{g}$, we construct an analogue of the Hochschild-Serre spectral sequence for Lie superalgebras.

Consider a short exact sequence of Lie superalgebras
$$0 \rightarrow \mathfrak{h} \rightarrow \mathfrak{g}\rightarrow \mathfrak{g}/\mathfrak{h} \rightarrow 0$$
and functors: $$\mathcal{F}:\mathfrak{g}/\mathfrak{h}\text{-mod} \rightarrow \mathbb{C}\text{-mod}$$
$$\mathcal{G}:\mathfrak{g}\text{-mod} \rightarrow \mathfrak{g}/\mathfrak{h}\text{-mod},$$
which are given by $\mathcal{F}(-) = \operatorname{H}^0(\mathfrak{g}/\mathfrak{h}, -)$ and $\mathcal{G}(-) = \operatorname{H}^0(\mathfrak{h}, -)$.  Both $\mathcal{F}$ and $\mathcal{G}$ satisfy the conditions given in \cite[Proposition 4.1]{J}, and so we obtain a Grothendieck spectral sequence:
$$E^{p,q}_2=R^p\mathcal{F}(R^q(\mathcal{G}(-))),$$
which converges to $R^{p+q}(\mathcal{FG})(-)$. As $\mathcal{F}\circ\mathcal{G} = \operatorname{H}^0(\mathfrak{g}, -)$, this simplifies to

$$E^{p,q}_2 = \operatorname{H}^p(\mathfrak{g}/\mathfrak{h}, \operatorname{H}^q(\mathfrak{h}, -)) \Rightarrow \operatorname{H}^{p+q}(\mathfrak{g}, -).$$

\subsection{Infinite families}
In this section, we provide a basis for $\mathfrak{n}$ for each of the infinite families of classical simple Lie superalgebras, and define an ideal $\mathfrak{I}$ of $\mathfrak{n}$.  As a consequence, for each family we will obtain a short exact sequence
$$0 \rightarrow \mathfrak{I} \rightarrow \mathfrak{n} \rightarrow \mathfrak{n/I} \rightarrow 0,$$
which will give rise to a Hochschild-Serre spectral sequence
$$E^{ij}_2 = \operatorname{H}^i(\mathfrak{n/I},\operatorname{H}^j(\mathfrak{I},\mathbb{C})) \Rightarrow \operatorname{H}^{i+j}(\mathfrak{n}, \mathbb{C}).$$
We then show in the following section that each of these spectral sequences collapses.

\subsubsection{$\mathfrak{gl}(m|n)$}

Let $\mathfrak{g} = \mathfrak{gl}(m|n)$ where $m \geq n$ and let $\mathfrak{n}^- \oplus \mathfrak{f} \oplus \mathfrak{n}$ be its triangular decomposition.  Following \cite[Section 1.1.2]{CW} we label the rows and columns of elements of $\mathfrak{gl}(m|n)$ by elements of the set $\{\bar{1}, \cdots \bar{m}, 1, \cdots n\}$.  We let $E_{ij}$ denote the elementary matrix for row $i$ and column $j$.  Then $\mathfrak{n}$ is spanned by

\[ \begin{cases} 
      E_{\bar{i},\bar{j}} \quad (\epsilon_i-\epsilon_j) & 1 \leq i < j \leq m \\
      E_{{i},{j}} \quad (\delta_i - \delta_j) & 1 \leq i < j \leq n \\
      E_{\bar{i},{j}} \quad (\epsilon_i-\delta_j)& 1 \leq i  \leq m, 1 \leq j \leq n, i < j \\
      E_{{i},\bar{j}} \quad (\delta_i-\epsilon_j) & 1 \leq i \leq n, 1 \leq j \leq m, i < j, \\
   \end{cases}
\]
where the quantity in parentheses denotes the corresponding weight under the action of the maximal torus.

We let $\mathfrak{I} \subseteq \mathfrak{n}$ be the subalgebra spanned by elements $E_{\bar{i},\bar{m}}$, $E_{\bar{i},n}$, $E_{i, \bar{m}}$, and $E_{i, n}$ in the case where $m=n$, and by just $E_{\bar{i},\bar{m}}$ and $E_{i, \bar{m}}$ when $m>n$, with the appropriate bounds on $i$.  Using the supercommutator identity:
$$[E_{ij},E_{kl}]=\delta_{jk}E_{il}-(-1)^{|E_{ij}|\cdot|E_{kl}|}\delta_{li}E_{jk},$$
it is a simple computation to show that $\mathfrak{I}$ is an ideal of $\mathfrak{n}.$

\subsubsection{$\mathfrak{osp}(2m+1|2n)$}

Let $m \geq n$.  We may view $\mathfrak{osp}(2m+1|2n)$ as being a subalgebra of $\mathfrak{gl}(2m+1|2n)$, and so we may describe its spanning set by means of the same elementary matrices.  In particular, $\mathfrak{osp}(2m+1|2n)$ will be the span of the root vectors and maximal torus as described in \cite[Section 1.2.4]{CW}.  Restricting our view to the weight spaces listed in the above table, let $\mathfrak{n}$ be the subalgebra whose odd component is spanned by the elements:
\[ \begin{cases} 
      E_{k+n,\overline{i+m}} + E_{\bar{i},k} \quad (-\epsilon_i+\delta_j)& \\
      -E_{\overline{i+m},k+n} + E_{k,\bar{i}} \quad (-\delta_i+\epsilon_j) & \\
      E_{k+n,\bar{l}} - E_{\overline{l+m},k} \quad (-\epsilon_k-\delta_l)& \\
      E_{2n+1,\bar{t}} + E_{\overline{t+m}, 2n+1} \quad (\delta_t),&\\
   \end{cases}
\]
where $1 \leq i \leq m$ and $1 \leq k \leq n$, and whose even component is the direct sum of the nilpotent radicals of $\mathfrak{so}(2m+1)$ and $\mathfrak{sp}(2n)$.

We let $\mathfrak{I}$ be the subalgebra of $\mathfrak{n}$ spanned by all root vectors with weights containing an $\epsilon_m$ a $\delta_n$ term.  Again, it may be shown that this constitutes an ideal of $\mathfrak{n}$.

\subsubsection{$\mathfrak{osp}(2m|2n)$}
The $\mathfrak{n}$ arising from $\mathfrak{osp}(2m|2n)$ has a similar basis as in the $\mathfrak{osp}(2m+1|2n)$ case, with an odd part given by:
\[ \begin{cases} 
      E_{k+n,\overline{i+m}} + E_{\bar{i},k} \quad (-\epsilon_i+\delta_j)& \\
      -E_{\overline{i+m},k+n} + E_{k,\bar{i}} \quad (-\delta_i+\epsilon_j) & \\
      E_{k+n,\bar{l}} - E_{\overline{l+m},k} \quad (-\epsilon_k-\delta_l)& \\
   \end{cases}
\]
and an even part given by the direct sum of the nilpotent radicals of $\mathfrak{so}(2m)$ and $\mathfrak{sp}(2n)$.

We may define an ideal just as we did for $\mathfrak{osp}(2m+1|2n)$, letting $\mathfrak{I}$ be the collection of all root vectors corresponding to weights of $\mathfrak{n}$ containing an $\epsilon_m$ term.

\subsubsection{$\mathfrak{q}(n)$}

We may view $\mathfrak{q}(n)$ as the subalgebra of $\mathfrak{gl}(n|n)$ spanned by the elements:
$$\widetilde{E}_{ij} := E_{\bar{i}\bar{j}}+E_{ij} \quad (\epsilon_i-\epsilon_j), \quad \overline{E}_{ij} := E_{i\bar{j}}+E_{\bar{i},j} \quad (\epsilon'_i-\epsilon'_j), \quad 1 \leq i,j \leq n.$$
Then $\mathfrak{n}$ is the subalgebra spanned by all $\widetilde{E}_{ij}$ and $\overline{E}_{ij}$ where $i<j$.  Let $\mathfrak{I}$ be the subalgebra of $\mathfrak{n}$ generated by all $\widetilde{E}_{in}$ and $\overline{E}_{in}$.  Again, it is not too difficult to show that $\mathfrak{I}$ is an ideal of $\mathfrak{n}$.

\subsection{Collapsing}

\begin{theorem}
For any of the infinite families of classical Lie superalgebras $\mathfrak{g}$, the corresponding spectral sequence $E^{ij}_r$ collapses on the $r=2$ page.
\end{theorem}
\begin{proof}
Recall that the differentials $d_r$ on the $r$th page of a spectral sequence have bidegree $(r,1-r)$, sending $E_r^{ij}$ to $E_r^{i+r, j-r+1}$.  Our goal is to show that for each page $r \geq 2$, the differentials must all be 0. First, note that we may decompose all $E_r^{ij}$ into a direct sum of weight spaces under the action of the maximal torus of $\mathfrak{f}$.  The differentials respect this action, and so to show that $d_r$ is identically 0, it is sufficient to show that no weight in $E_r^{ij}$ appears in $E_r^{i+r, j-r+1}$.  To demonstrate this, we split the proof up into different cases for each classical superalgebra.

\begin{enumerate}

\item{$\mathfrak{gl}(m|n)$}
Consider an arbitrary differential from the $E_2$ page: $d_2:E^{ij}_2 \rightarrow E^{i+2,j-1}_2$.

The term $E^{ij}_2$ is a subquotient of $\Lambda^i_s(\mathfrak{n/I})^* \otimes \Lambda^j_s(\mathfrak{I})^*$, and so any weight of $E^{ij}$ must also be a weight of $\Lambda^i_s(\mathfrak{n/I})^* \otimes \Lambda^j_s(\mathfrak{I})^*$.  As the weights of $\mathfrak{n/I}$ are of the form $\epsilon_k - \delta_l$, $\delta_k - \epsilon_l$, $\epsilon_k-\epsilon_l$ and $\delta_k-\delta_l$ for $1<k<l<n$ and the weights of $\mathfrak{I}$ are of the form $\epsilon_i -\delta_n$, $\delta_i-\delta_n$, $\epsilon_i-\epsilon_n$ and $\delta_i-\epsilon_n$ for $1<i<n$, the weights of $E^{ij}$ all have $j$ summands containing either $\epsilon_n$ or $\delta_n$.  As the weights of $E^{i+2,j-1}$ have only $j-1$ such summands, $d_2$ must be the zero map.

We therefore have that $E^{ij}_3 = E_2^{ij}$ for all $i$ and $j$.  However, we can apply the same argument to the differentials on the $E_r$ page for any arbitrary $r$.  Namely, if the weights in the domain of $d_r$ have j copies of $\epsilon_m$ or $\delta_n$, then those in the image have only $j-r$ such copies.  Thus, $d_r$ must again be the 0 map.  Thus for all $r>2$, $E^{ij}_r=E^{ij}_2$, and so the spectral sequence collapses.

\item{$\mathfrak{sl}(m|n)$}
The collection of weights corresponding to the $\mathfrak{n}$ in $\mathfrak{sl}(m|n)$ are identical to those for $\mathfrak{gl}(m|n)$.  Hence, we may take the same ideal of $\mathfrak{n} \subseteq \mathfrak{sl}(m|n)$ and the same spectral sequence will collapse.

\item{$\mathfrak{osp}(2m+1|2n)$}
The ideal $\mathfrak{I}$ is spanned by all weight spaces of a root containing  $\epsilon_m$.  Thus an arbitrary weight of $E^{pq}_r$ must have a total of q copies of or $\epsilon_m$, whereas those in $E^{p+r, q+(1-r)}_r$ have only $q+1-r$ copies.  Thus any differential $d_r$ must be 0, and so the spectral sequence collapses on the $E_2$ page.

\item{$\mathfrak{osp}(2m|2n)$}
We defined the ideal for $\mathfrak{osp}(2m|2n)$ similarly to how it was defined for $\mathfrak{osp}(2m+1|2n)$, and so the above argument follows in the same way.

\item{$\mathfrak{q}(n)$}
As $E^{ij}_2$ is a subquotient of $\Lambda^i_s(\mathfrak{n/I}^*) \otimes \Lambda^j_s\mathfrak{I}^*$, all of its weights must contain $j$ total summands containing either copy of $\epsilon_n$, whereas $E^{i+r, j+1-r}_2$ only contains $j+1-r$ such copies, and thus an arbitrary differential $d_r:E^{ij}_2 \to E^{i+2,j-1}_2$ must be 0, so the spectral sequence again collapses.

\end{enumerate}
\end{proof}

\section{$\operatorname{H}^1(\mathfrak{n}, \mathbb{C})$ Cohomology}

\subsection{Superderivations}

It is well known that in the case of ordinary Lie algebras, $\operatorname{H}^1(\mathfrak{g},M)$ corresponds to derivations from $\mathfrak{g}$ to $M$ modulo inner derivations \cite{HS}.  This situation generalizes to the Lie superalgebra case.

We define a \textit{superderivation} from a Lie superalgebra $\mathfrak{g}$ to a $\mathfrak{g}$-module $M$ to be a linear map $\phi$ satisfying
$$\phi([xy]) = x \cdot \phi(y) - (-1)^{|x||y|} y \cdot \phi(x).$$
An \textit{inner superderivation} is a derivation of the form $\phi_a(x) = x \cdot a$ for some $a \in M$.

\begin{proposition}
$\supder(\mathfrak{g},M) \cong \Hom(I\mathfrak{g},M)$.
\end{proposition}

\begin{proof}
Let $d: \mathfrak{g} \to M$ be a superderivation.  Consider the map $f'_d:T(\mathfrak{g}) \to M$ given by $f_d(x_1 \otimes \cdots \otimes x_n) = x_1 \circ \cdots \circ d(x_n)$ and which sends $T^0(\mathfrak{g})$ to 0.  It follows immediately that $f'_d$ vanishes on $I$ and thus defines a morphism on $U(\mathfrak{g})$ which restricts to a homomorphism $f_d:I\mathfrak{g} \to M$.

Conversely, given a homomorphism $f:I\mathfrak{g} \to M$, we can extend it to a map on all of $U(\mathfrak{g})$ by setting $f(T^0(\mathfrak{g})) = 0$ and letting $d_f = f \circ i$.  It is straightforward to show that $f_{d_f} = f$ and $d_{f_d} = d$, and so the map sending $f$ to $d_f$ is an isomorphism between $\supder(\mathfrak{g},M)$ and $\Hom(I\mathfrak{g},M)$.
\end{proof}

\begin{proposition}
$\operatorname{H}^1(\mathfrak{g},M) \cong \supder(\mathfrak{g},M)/\innsupder(\mathfrak{g},M).$
\end{proposition}

\begin{proof}
From the augmentation map, we obtain the following short exact sequence:
$$0 \rightarrow I\mathfrak{g} \rightarrow U(\mathfrak{g}) \rightarrow \mathbb{C} \rightarrow 0.$$
From the corresponding long exact sequence in cohomology, we obtain that
$$\operatorname{H}^1(\mathfrak{g},M) \cong \coker(\Hom(U(\mathfrak{g}),M) \rightarrow \Hom(I\mathfrak{g},M)) \cong \supder(\mathfrak{g},M)/\im(\Hom(U(\mathfrak{g}),M)).$$
However, if $f \in \Hom(U(\mathfrak{g}), M)$, and $f(1) = a$, then the corresponding derivation is $d_f(x) = x \cdot a$, and thus $\operatorname{H}^1(\mathfrak{g},M) \cong \supder(\mathfrak{g},M)/\innsupder(\mathfrak{g},M).$

\end{proof}

In particular, when using trivial coefficients, we have the following result:

\begin{theorem}
$\operatorname{H}^1(\mathfrak{g}, \mathbb{C}) \cong (\mathfrak{g}/[\mathfrak{g},\mathfrak{g}])^*$.
\end{theorem}

\subsection{Explicit calculations}
By the above theorem, to compute the first cohomology, it is sufficient to describe both $\mathfrak{n}$ and $[\mathfrak{n}, \mathfrak{n}]$. As we have already provided bases for $\mathfrak{n}$ in Section 3, below we do the same for $[\mathfrak{n}, \mathfrak{n}]$ and give formulas for the dimensions of $\mathfrak{n}$, $[\mathfrak{n}, \mathfrak{n}]$, and $\operatorname{H}^1(\mathfrak{n}, \mathbb{C})$.  A table of corresponding weights is given in Section 6.

\subsubsection{$\mathfrak{gl}(m|n)$}

We have that the elementary matrices $E_{ij}$ that span $\mathfrak{n}$ will be in $[\mathfrak{n}, \mathfrak{n}]$ precisely when $j-i \geq 2$, and so $[\mathfrak{n}, \mathfrak{n}]$ will have a basis given by
\[ \begin{cases} 
      E_{\bar{i},\bar{j}} & 1 \leq i ,j \leq m, j-i \geq 2 \\
      E_{{i},{j}} & 1 \leq i , j \leq n, j-i \geq 2 \\
      E_{\bar{i},{j}} & 1 \leq i  \leq m, 1 \leq j \leq n, j-i \geq 2 \\
      E_{{i},\bar{j}} & 1 \leq i \leq n, 1 \leq j \leq m, j-i \geq 2.\\
   \end{cases}
\]
The Lie superalgebra $\mathfrak{n}$ has dimension ${m\choose2} + n \cdot (m-n) + 3\cdot {n\choose 2}$ and $[\mathfrak{n}, \mathfrak{n}]$ has dimension $${m-1\choose 2}+2 \cdot {n-1 \choose 2} + n \cdot (m-n-1) + {n \choose 2},$$ and so $\operatorname{H}^1(\mathfrak{n}, \mathbb{C})$ has dimension $m-1+n-1+n-1+n = m+3n-3$.  The weights of $\operatorname{H}^1(\mathfrak{n}, \mathbb{C})$ can be found by using the information listed in the previous section and are included in the tables in Section 6.

\subsubsection{$\mathfrak{sl}(n|n)$}
The weight space decomposition for $\mathfrak{n}$ is identical to that in the $\mathfrak{gl}(n|n)$ case, and thus the above dimension formula and weight space decomposition hold.

\subsubsection{$\mathfrak{osp}(2m|2n)$}
The derived subalgebra $[\mathfrak{n}, \mathfrak{n}]$ is spanned by the elements
\[ \begin{cases}
      E_{j,\overline{i}} - E_{\overline{i+n},j+m} \\
      E_{j+m,\overline{i+n}} - E_{\overline{i},j} \\
      E_{l,\overline{k+n}} - E_{\overline{k+n},l+m} \\
      E_{\overline{i}, \overline{i+n}} \\
      E_{\overline{i},\overline{k+n}}+E_{\overline{k},\overline{i+n}}\\
      E_{\overline{i},\overline{k}}-E_{\overline{k+n},\overline{i+n}}\\
      E_{j,l+m}-E_{l,j+m}\\
      E_{jl}-E_{l+m,j+m},\\
   \end{cases}
\]
where $1 \leq i,k \leq n$ and $1 \leq j,l \leq m$ and $j-i \geq 2$.
The quotient by this subalgebra consists of root vectors solely with the corresponding weights $\epsilon_j-\delta_{j+1}$, $\delta_j-\epsilon_{j+1}$, $\epsilon_m+\delta_n$, $2\delta_n$, $\delta_i-\delta_{i+1}$, and $\epsilon_i-\epsilon_{i+1}$.  As a result, $\H^1(\mathfrak{n}, \mathbb{C})$ has dimension $$ 2(m-1)+2(n-1)+2=2m+2n-2.$$

\subsubsection{$\mathfrak{osp}(2m+1|2n)$}
The only difference in terms of dimension between this and the preceding case is the existence of a root in $\mathfrak{n}$ not found in $[\mathfrak{n}, \mathfrak{n}]$.  Thus, the dimension calculation may proceed in essentially the same way, yielding a dimension formula of $2m+2n-1$.

\subsubsection{$\mathfrak{q}(n)$}
Much like in the case of $\mathfrak{gl}(m|n)$, if $\mathfrak{g} = \mathfrak{q}(n)$, then $[\mathfrak{n}, \mathfrak{n}]$ is spanned by the matrices:

\[ \begin{cases} 
      \overline{E}_{{i},{j}} & 1 \leq i , j \leq n, \, j-i \geq 2 \\
      \widetilde{E}_{{i},{j}} & 1 \leq i ,  j \leq n,\, j-i \geq 2 \\
   \end{cases}
\]

Hence, the dimension of $[\mathfrak{n},\mathfrak{n}]$ is $2 \cdot {n-1 \choose 2} = (n-1)(n-2)$.  As the dimension of $\mathfrak{n}$ is $2 \cdot {n \choose 2} = n(n-1)$, this implies that the dimension of $\operatorname{H}^1(\mathfrak{n}, \mathbb{C})$ is $$n(n-1)-(n-2)(n-1) = 2(n-1).$$

\subsubsection{$D(2,1,\alpha)$, $G(3)$, and $F(4)$}
For each of the exceptional superalgebras, we may look at the weights given in the table from Section 2.  As no two weights add up to a third, it follows that the bracket is 0 on $\mathfrak{n}_{\bar{1}}$ and so that $\mathfrak{n}_{\bar{1}}$ is abelian, and so $\mathfrak{n} \cong \mathfrak{n}/[\mathfrak{n}, \mathfrak{n}]$.

\section{$\operatorname{H}^2(\mathfrak{n}, \mathbb{C})$ Cohomology}

\subsection{Central Extensions}
As in the case of $\operatorname{H}^1(\mathfrak{n}, \mathbb{C})$, the classical Lie algebra interpretation of equivalence classes of extensions extends to the superalgebra case.  On the cochain complex $C^n(\mathfrak{g},M)$ we set the following $\mathbb{Z}_2$ grading:
$$C^n(\mathfrak{g},M)_{\alpha} = \{f \in \Hom(\Lambda^n_s(\mathfrak{g},M)  | f(\Lambda^n(\mathfrak{g}))_{\beta} \}\subseteq M_{\alpha+\beta},$$
where $\alpha$ and $\beta$ are elements of $\mathbb{Z}_2$. As the differential map preserves this grading, this gives rise to a $\mathbb{Z}_2$ grading on $\operatorname{H}^n(\mathfrak{g},M)$ as well.

If $M$ is a $\mathfrak{g}$-module, regarding $M$ as an abelian superalgebra, we say that $\mathfrak{h}$ is an extension of $\mathfrak{g}$ by $M$ if there is an exact sequence of $\mathfrak{g}$-modules:
$$0 \rightarrow M \rightarrow \mathfrak{h} \rightarrow \mathfrak{g} \rightarrow 0,$$
where $\mathfrak{h}$ is a Lie superalgebra.
Two such extensions are said to be equivalent if there is a commutative diagram

\[ \begin{tikzcd}
0 \arrow{r}& M \arrow{r}{\varphi} \arrow[swap]{d}{id} & \mathfrak{h} \arrow{r} \arrow{d} & \mathfrak{g} \arrow{d}{id} \arrow{r} & 0 \\%
0 \arrow{r} &M \arrow{r}{\varphi_f}& \mathfrak{h} \arrow{r}& \mathfrak{g} \arrow{r} &0
\end{tikzcd}.
\]
Given an even cocycle $h$, we define the extension $E_h$ via the short exact sequence
$$0 \rightarrow M \rightarrow \mathfrak{g} \oplus M \rightarrow \mathfrak{g} \rightarrow 0,$$
where the product in $\mathfrak{g} \oplus M$ is given by
$$[(x,m),(y,n)] = ([x,y],xn-(-1)^{|m||y|}ym + h(x,y)).$$

Every extension will be equivalent to $E_h$ for some even cocycle $h$.  Moreover, one can show that two extensions $E_h$ and $E_{h'}$ are equivalent if and only if there is some even linear map $f: \mathfrak{g} \to M$ such that $df = h - h'$, and thus the equivalence classes of extensions are in one-to-one correspondence with $\operatorname{H}^2(\mathfrak{g},M)_{\bar{0}}$ \cite[Section 16.4]{M}.

\subsection{Computing $\operatorname{H}^2$}
Computing the $\operatorname{H}^2(\mathfrak{n}, \mathbb{C})$ cohomology involves a term mixing together both odd and even elements, and thus requires much more care than the $\operatorname{H}^1$ case.  The main idea will be to compute the dimension of these groups recursively.  For simplicity's sake, let us restrict our attention to $\mathfrak{g}=\mathfrak{gl}(n|n)$, and let $\mathfrak{n}(n)$ denote the corresponding nilpotent radical.  From the collapsing of Hochschild-Serre spectral sequence, we have that:
\begin{equation}\label{e:directsum}
\operatorname{H}^2(\mathfrak{n}(n), \mathbb{C}) \cong \operatorname{H}^0(\mathfrak{n}(n)/\mathfrak{I}, \operatorname{H}^2(\mathfrak{I}, \mathbb{C})) \oplus  \operatorname{H}^1(\mathfrak{n}(n)/\mathfrak{I}, \operatorname{H}^1(\mathfrak{I}, \mathbb{C})) \oplus  \operatorname{H}^2(\mathfrak{n}(n)/\mathfrak{I}, \operatorname{H}^0(\mathfrak{I}, \mathbb{C})),
\end{equation}
where $\mathfrak{I}$ is the ideal described in Section 3.  As $\mathfrak{I}$ is abelian, the cohomology groups $\operatorname{H}^n(\mathfrak{I}, \mathbb{C})$ can be easily computed.  Additionally, there is a natural isomorphism between $\mathfrak{n}(n)/\mathfrak{I}$ and $\mathfrak{n}(n-1)$.  Thus, in the above decomposition, the first term can be computed directly, viewing it as the set of fixed points of $\operatorname{H}^2(\mathfrak{I},\mathbb{C})$ under the action of $\mathfrak{n}(n-1)$, and the third can be computed recursively.  Thus, the main issue is the computation of $\operatorname{H}^1(\mathfrak{n}(n)/\mathfrak{I}, \operatorname{H}^1(\mathfrak{I}, \mathbb{C}))$, which is isomorphic to $\operatorname{H}^1(\mathfrak{n}(n-1), \mathfrak{I}^*)$.
\subsection{Low-Dimension Examples}
As an example where all of the computations are relatively straightforward, let us first consider the case of $\mathfrak{gl}(2|2)$ where we wish to compute $\operatorname{H}^2(\mathfrak{n}(2),\mathbb{C})$.  As $\mathfrak{n}(2)$ is abelian, all of the differentials in the cochain complex
$$C^0 \rightarrow C^1 \rightarrow C^2 \rightarrow \cdots$$
are 0, where $C^i \cong \Lambda_s^i(\mathfrak{n}(2)^*)$.  As such, for any $i$, $\operatorname{H}^i(\mathfrak{n}(2), \mathbb{C}) \cong \Lambda_s^i(\mathfrak{n}(2)^*)$.  In particular,
$$\operatorname{H}^2(\mathfrak{n}, \mathbb{C}) \cong \Lambda_s^2(\mathfrak{n}^*) \cong \bigoplus_{i+j=2}\Lambda^i(\mathfrak{n}_{\bar{0}}) \otimes \operatorname{S}^j(\mathfrak{n}_{\bar{1}}).$$
Using the formulas for the dimensions of exterior and symmetric algebras on a vector space of dimension $n$, namely
$$\dim \Lambda^i(V) = \binom{n}{i}$$
and $$\dim \operatorname{S}^j(V) = \binom{n+j-1}{j},$$
we obtain $$\dim \operatorname{H}^2(\mathfrak{n}, \mathbb{C}) = \dim \Lambda_s^2(\mathfrak{n}) = 1 \cdot 3 + 2 \cdot 2 + 1 \cdot 1 = 8.$$
Now consider the case where $\mathfrak{g} = \mathfrak{gl}(3|3)$, and we wish to compute $\operatorname{H}^2(\mathfrak{n}(3),\mathbb{C})$.
Letting $\mathfrak{n}$ denote $\mathfrak{n}(2)$, note that as $\mathfrak{n}$ is abelian, $\mathfrak{n}_{\bar{0}}$ is an ideal of $\mathfrak{n}$, and so we obtain a short exact sequence
$$0 \rightarrow \mathfrak{n}_{\bar{0}} \rightarrow \mathfrak{n} \rightarrow \mathfrak{n}_{\bar{1}} \rightarrow 0.$$
This gives rise to a second Hochschild-Serre spectral sequence:
$$\overbar{E^{i,j}_2} = \operatorname{H}^i(\mathfrak{n}_{\bar{1}}, \operatorname{H}^j(\mathfrak{n}_{\bar{0}}, \mathfrak{I}_{\bar{0}}^* \otimes \mathfrak{I}_{\bar{1}}^*)) \Rightarrow \operatorname{H}^{i+j}(\mathfrak{n}, \mathfrak{I}_{\bar{0}}^* \otimes \mathfrak{I}_{\bar{1}}^*).$$
Again appealing to an argument with weights, the differential $d^2$ sends  $\overbar{E^{0,1}_2}$ to 0.  As the spectral sequence is in the first quadrant, all subsequent differential must do the same.
Thus, we have that
$$\operatorname{H}^1(\mathfrak{n},\mathfrak{I}_{\bar{0}}^* \otimes \mathfrak{I}_{\bar{1}}^*) \cong \overbar{E^{0,1}_2} \oplus \overbar{E^{1,0}_2}$$ 
As $\overbar{E^{1,0}_2} = \operatorname{H}^1(\mathfrak{n}_{\bar{1}}, \operatorname{H}^0(\mathfrak{n}_{\bar{0}}, \mathfrak{I}_{\bar{0}}^* \oplus \mathfrak{I}_{\bar{1}}^*)) = \operatorname{H}^1(\mathfrak{n}_{\bar{1}}, \mathbb{C}^{\oplus4})$, we can simplify this as $\operatorname{H}^1(\mathfrak{n}_{\bar{1}})^{\oplus4}$.  As $\mathfrak{n}_{\bar{1}}$ is abelian of dimension 2, $\overbar{E^{1,0}_2}$ must have dimension 8.
On the other hand, $\overbar{E^{0,1}_2} \cong \operatorname{H}^0(\mathfrak{n}_{\bar{1}}, \operatorname{H}^1(\mathfrak{n}_{\bar{0}}, \mathfrak{I}^*))$.  However, as $\mathfrak{n}_{\bar{0}}$ is a classical Lie algebra, by Kostant's theorem,
$$\operatorname{H}^1(\mathfrak{n}_{\bar{0}}, \mathfrak{I}^*) \cong \bigoplus_{l(w)=1, \, j \in J} w \cdot \lambda_j,$$
where $w$ is an element of the Weyl group of $\mathfrak{n}_{\bar{0}}$ and $\mathfrak{I}^* = \bigoplus_{j \in J} \operatorname{L}(\lambda_j)$ as a direct sum of $\mathfrak{n}_{\bar{0}}$ modules.  (Viewing $\mathfrak{I}^*$ as an $\mathfrak{sl}(2) \times \mathfrak{sl}(2)$-module shows it is isomorphic to $\operatorname{L}((1,0)) \oplus \operatorname{L}((1,0)) \oplus \operatorname{L}((0,1)) \oplus \operatorname{L}((0,1)).)$
As the Weyl group of $\mathfrak{n}_{\bar{0}}$ is isomorphic to  $\Sigma_2 \times \Sigma_2$, there are 2 elements of length 1, and so $\overbar{E^{0,1}_2} = \operatorname{H}^0(\mathfrak{n}_{\bar{1}}, s_\alpha \cdot \mathfrak{I}^*)$, which has dimension 4. Thus, altogether $\operatorname{H}^1(\mathfrak{n}_{\bar{0}}, \mathfrak{I}^*)$ has dimension 8, from which an easy computation shows that the dimension of the set of fixed points under the action of $\mathfrak{n}_{\bar{1}}$ is 4, which implies $\operatorname{H}^1(\mathfrak{n}, \mathfrak{I}^*)$ to have a total dimension of 12. Using the argument below, we can see that $\operatorname{H}^0(\mathfrak{n}, \Lambda_s^2(\mathfrak{I}^*))$ has dimension 8 and we already know 
$\operatorname{H}^2(\mathfrak{n}, \mathbb{C})$ has dimension 8, so altogether, this implies that $\operatorname{H}^2(\mathfrak{n}(3), \mathbb{C})$ has dimension 28.  However, the argument for computing the dimension of $\operatorname{H}^1(\mathfrak{n}, \mathfrak{I}^*)$ was only valid because $\mathfrak{n}$ was abelian.  For general $\mathfrak{gl}(n|n)$ this isn't the case, so $\mathfrak{n}_{\bar{0}}$ is not necessarily an ideal of $\mathfrak{n}$.

\subsection{Explicit Calculations}

\subsubsection{$\mathfrak{gl}(n|n)$}
Before beginning with the more general case of $\mathfrak{gl}(m|n)$, we start with the more special case of $\mathfrak{gl}(n|n)$.
As in the general case above, we may compute $\operatorname{H}^2(\mathfrak{n}, \mathbb{C})$ by means of the direct sum decomposition from the spectral sequence, i.e.,
$$\operatorname{H}^2(\mathfrak{n}, \mathbb{C}) \cong \operatorname{H}^0(\mathfrak{n}/\mathfrak{I}, \Lambda_s^2(\mathfrak{I}^*)) \oplus \operatorname{H}^1(\mathfrak{n}/\mathfrak{I}, \mathfrak{I}^*) \oplus \operatorname{H}^2(\mathfrak{n}/\mathfrak{I}, \mathbb{C}).$$
The first term can be identified with the set of fixed points of $\Lambda_s^2(\mathfrak{I}^*)$ under the action of $\mathfrak{n}/\mathfrak{I}$, i.e., all $x \in \Lambda_s^2(\mathfrak{I}^*)$ such that $(\mathfrak{n}/\mathfrak{I}) \cdot x = 0$.  This set is not particularly difficult to calculate, and we get the following result.

\begin{proposition}
For all $n$, $\operatorname{H}^0(\mathfrak{n}/\mathfrak{I}, \Lambda_s^2(\mathfrak{I}^*))$ has dimension 8.
\end{proposition}
\begin{proof}
Note that if $a \in \mathfrak{n}/\mathfrak{I}$ and $x \in \Lambda_s^2(\mathfrak{I}^*)$ are weight vectors of weights $\lambda$ and $\mu$,  then $a \cdot x$ has weight $\lambda+\mu$, and so $a$ sends distinct weight spaces to distinct weight spaces.  In particular, if $x_1 + \cdots + x_n$ is a sum of weight vectors of distinct weights in $\Lambda_s^2(\mathfrak{I}^*)$ and $a \cdot (x_1 + \cdots + x_n)=0$, then $a\cdot x_i$ must equal $0$ for all $i$.  Since the standard basis for $\Lambda_s^2(\mathfrak{I}^*)$ consists of root vectors all of distinct weights, it suffices to look at which basis elements are sent to 0 by $\mathfrak{n}/\mathfrak{I}$.

$\mathfrak{I}^*$ has a basis given by $E_{i,n}^*$, $E_{\overline{i},n}^*$, $E_{i,\overline{n}}^*$, and $E_{\overline{i},\overline{n}}^*$, for $1 \leq i \leq n-1$.  Based on the supercommutator identity, if $E_{i,j}$ is in $ \mathfrak{n}/\mathfrak{I}$ and $E_{k,n}^*$ or $E_{k, \overline{n}}^*$ is in $\mathfrak{I}^*$, $E_{i,j} \cdot E_{k,n}^*$ doesn't vanish precisely when $i=k$.  In particular, as there are no elements $E_{i,j}$ in $\mathfrak{n}/\mathfrak{I}$ where $i=n-1$ or $\overline{n-1}$, it is precisely the basis elements $E_{n-1,n}^*$, $E_{\overline{n-1},n}^*$, $E_{n-1,\overline{n}}^*$, and $E_{\overline{n-1},\overline{n}}^*$ that are sent to 0 for all $a \in \mathfrak{n}/\mathfrak{I}$.  Any element of $\Lambda_s^2(\mathfrak{I}^*)$ that is sent to 0 is the superexterior product of two such basis elements of $\mathfrak{I}^*$, and as there are two even and two odd such basis elements, viewing $\Lambda^2_s(\mathfrak{I}^*)$ as $\Lambda^2(\mathfrak{I}_{\bar{0}}^*) \oplus (\mathfrak{I}^*_{\bar{0}} \otimes \mathfrak{I}^*_{\bar{1}}) \oplus \operatorname{S}^2(\mathfrak{I}^*_{\bar{1}})$, the total dimension of $\operatorname{H}^0(\mathfrak{n}/\mathfrak{I}, \Lambda_s^2(\mathfrak{I}^*))$ is $1 +2 \cdot 2 + 3=8$.
\end{proof}

Moreover, the third term may be computed recursively, using the fact that $\mathfrak{n}/\mathfrak{I}$ is isomorphic to $\mathfrak{n}$ from $\mathfrak{gl}(n-1|n-1)$.  Thus, it remains to compute the middle term.

Let us consider the cochain complex
$$ C^0 \rightarrow C^1 \rightarrow C^2 \rightarrow \cdots,$$
where $C^i \cong \Lambda_s^i(\mathfrak{n}/\mathfrak{I}^*) \otimes \mathfrak{I}^*$ and where the differentials are as in the introduction.  Then the middle term $\operatorname{H}^1(\mathfrak{n}/\mathfrak{I}, \mathfrak{I}^*)$ is given by the cohomology of the complex at $C^1$.  Since the differentials preserve the action of the torus, it follows that we may break up $C^i$ into its weight spaces.  The weights of $(\mathfrak{n}/\mathfrak{I})^*$ are of the form $\alpha_j - \beta_k$, where $\alpha$ and $\beta$ correspond to either $\epsilon$ or $\delta$, and $k<j<n$.  The weights of $\mathfrak{I}^*$ are of the form $\alpha'_n - \beta'_i$, where $i<n$.  All weights of $C^i$ will be sums of weights of these forms.  Actually, using the fact that the cohomology will be a subquotient of $\mathfrak{n}/\mathfrak{I}/[\mathfrak{n}/\mathfrak{I}, \mathfrak{n}/\mathfrak{I}]^* \otimes \mathfrak{I}^*$, we need only consider those weights of $(\mathfrak{n}/\mathfrak{I})^*$ of the form $\alpha_{j+1} - \beta_j$.  As a shorthand, given a weight $\alpha_i-\beta_j$, we let $F_{i,j}$ and $G_{i,j}$ denote the basis vector of $(\mathfrak{n}/\mathfrak{I})^*$ of weight $\alpha_i-\beta_j$, or more explicitly:

\[ F_{i,j}, \, G_{i,j} =\begin{cases} 
      E_{\bar{j},\bar{i}}^* & \alpha=\epsilon, \, \beta=\epsilon \\
      E_{{j},{i}}^* & \alpha=\delta, \, \beta=\delta \\
      E_{{j},\bar{i}}^* & \alpha=\epsilon, \, \beta=\delta \\
     E_{\bar{j},{i}}^* & \alpha=\delta, \, \beta=\epsilon. \\
   \end{cases}
\]

\begin{proposition}
The dimension for a weight space of $C^1$ is at most 2.
\end{proposition}
\begin{proof}
Suppose two basis vectors for $C^1$, $F_{j+1,j} \otimes G_{n,k}$ with weight $(\alpha_{j+1}-\beta_j)+(\alpha'_n-\beta'_k)$ and  $F'_{l+1,l} \otimes G'_{n,m}$ with weight  $(\zeta_{l+1}-\eta_{l})+(\zeta'_n-\eta'_m)$ actually had the same weight.  As the $\epsilon_i$, $\delta_j$ are linearly independent, any weight has a unique representation as a sum of $\epsilon_i$'s and $\delta_j$'s.  This leads to two cases:
\begin{enumerate}
    \item If $\alpha_{j+1}$, $\beta{j}$, $\alpha'_n$, and $\beta'_k$ are all distinct, these must be, in some order, the same weights as $\zeta_{l+1}$, $\eta_{l}$, $\zeta'_n$, and $\eta'_m$.  Since $l+1<n$, it follows that $\zeta'_n=\alpha'_n$ and $\alpha_{j+1}=\zeta_{l+1}$, so $j=l$.  Thus, either $\eta_{l}$ equals either $\beta_j$ or $\beta'_k$, which leads to two possible basis vectors of the same weight, giving a total dimension of at most 2.
    \item If $\alpha_{j+1}=\beta'_k$, then $\zeta'_n=\alpha'_n$, $\eta_l=\beta_l$ and $\zeta_{l+1}=\eta'_m$.  Since this forces $l+1$ to equal $j+1$ and $m$ to equal $l+1$, $\zeta_{l+1}$ can equal only $\epsilon_{j+1}$ or $\delta_{j+1}$, which yields at most two basis vectors.
\end{enumerate}
\end{proof}

With this in mind, we aim to determine the dimension of the image of $d^0$ and kernel of $d^1$.  To do this, we will determine which weights appear in both $C^0$ and $C^1$ and which appear in $C^1$ but not $C^2$.  For the former calculation, to calculate the dimension of the image of $d^0$, first note that since its image is in $C^1$, the differential defined in Equation~\ref{e:differential} simplifies to
\begin{equation*}
\begin{aligned}
d^0f(\omega_0){} =  (-1)^{\tau_i}\omega_0 \cdot f(1), \\
\end{aligned}
\end{equation*}
where a function $f: \mathbb{C} \rightarrow \mathfrak{I}^*$ is identified with an element of $\mathfrak{I}^*$ via the map sending $f$ to $f(1)$.  What this means is that so long as there exists an element $x$ of $\mathfrak{n}\mathfrak{I}$ such that $x \cdot f(1) \neq 0$, then $d^0$ does not map $f$ to $0$.  If $f(1) \in \mathfrak{I}^*$ and $x \in \mathfrak{n}/\mathfrak{I}$ are nonzero weight vectors, this condition holds if the sum of the weights of $f(1)$ and $x$ is again a weight of $\mathfrak{I}^*$.A weight $\alpha'_n-\beta'_k$ of a basis vector $G_{n,k}$ of $\mathfrak{I}^*$ may be written as a weight in $C^1$ precisely when $k<n-1$.  In particular, $G_{n,k}$ will map to an element in the linear span of the root vectors $F'_{k+1, k} \otimes G'_{n,k+1}$ and $F''_{k+1,k} \otimes G''_{n,k+1}$ corresponding to $(\epsilon_{k+1}-\beta'_k)+(\alpha'_n-\epsilon_{k+1})$ and $(\delta_{k+1}-\beta'_k)+(\alpha'_n-\delta_{k+1})$, respectively.  
Since the differential preserves weights, and $\mathfrak{I}^*$ has $4(n-1)-4=4(n-2)$ weights of the above form, the dimension of the image of $d^0$ is $4 \cdot (n-2).$

To compute the dimension of the kernel, we rely heavily on the differential defined in Equation~\ref{e:differential} and note that a generic weight will be of the form $\alpha_{j+1}-\beta_j +\alpha'_n-\beta'_i$, where $j<n-1$.  So long as $i<n-1$, this weight may be written as $(\alpha'_n-\alpha'_{i+1})+(\alpha'_{i+1}-\beta'_i)+(\alpha_{j+1}-\beta_j)$, and so the differential will send the weight vector corresponding to $(\alpha_{j+1}-\beta{j})+(\alpha'_n-\beta'_i)$ to a nonzero element of $C^2$.  Thus the only weight vectors in the kernel have $\mathfrak{I}^*$ component with $i=n-1$.  There are four basis elements of $\mathfrak{I}^*$ with $i=n-1$ and there are $4(n-2)$ basis elements of $(\mathfrak{n}/\mathfrak{I}/[\mathfrak{n}/\mathfrak{I},\mathfrak{n}/\mathfrak{I}])^*$, so the one-dimensional weight spaces in the kernel contribute total dimension $4(n-2) \cdot 4 = 16(n-2)$. Note, however, that none of these elements are in the image of $d^0$.  Besides those corresponding to weights $\alpha'_n-\beta'_{n-1}$, which are already included in the span of the root vectors listed above, each of these adds 1 more dimension to the kernel.  As there are $4(n-3)$ such elements, this gives the kernel a total dimension of at least $20(n-2)-4$.

To show that no other elements are in the kernel, let $F_{j+1,j} \otimes G_{n,j}$ of weight  $(\alpha_{j+1}-\epsilon_j)+(\alpha'_n-\delta_j)$ and
$F'_{j+1,j} \otimes G'_{n,j}$ of weight $(\alpha_{j+1}-\delta_j)+(\alpha'_n-\epsilon_j)$ be two basis vectors of the same weight, where $j<n-1$.  Identify these basis elements with functions $f$ and $g$ from $\mathfrak{n}/\mathfrak{I}$ to $\mathfrak{I}$.  Using the action of the differential, we see that $df$ will send the element $F_{j+1, j} \wedge H_{j+1,j}$ of weight $\alpha_{j+1}-\epsilon_j + \beta_{j+1}-\delta_j$ to a root vector of weight $\alpha'_n-\beta_{j+1}$, where $\beta$ is either $\epsilon$ or $\delta$, depending on what $\alpha$ is not.  However, $dg$ will send the same element to 0.  Similarly, $dg$ will send $F'_{j+1,j} \wedge \bar{H}_{j+1,j}$ of weight $\alpha_{j+1} -\delta_j + \beta_{j+1}-\epsilon_j$ to $\alpha'_n-\beta_{j+1}$ while $df$ sends the same element to 0.  As $df$ and $dg$ are nonzero on different subsets of the basis elements, it follows that they must be linearly independent, and hence there is no nontrivial linear combination of $df$ and $dg$ equal to 0.  Since $f$ and $g$ span their weight space, any other nonzero element of that weight space gets mapped to a linear combination of $df$ and $dg$, and so cannot be mapped to 0 and is thus not in the kernel.  Therefore, any weight of the form $(\alpha_{j+1}-\epsilon_j)+(\alpha'_n-\delta_j)$ does not appear in the kernel and thus the kernel must have dimension of exactly $20(n-2)-4$, and so the dimension of the first cohomology is $$ \dim \ker{d^1}-\dim \im{d^0} = 20(n-2)-4-4(n-2)=16(n-2)-4.$$
Combining this with the fact that the first term in the direct sum decomposition above has dimension 8, we have that when $n>2$, the dimension of $\operatorname{H}^2(\mathfrak{n}, \mathbb{C})$ equals
$$8+ \sum_{i=3}^n(16(i-2)-4+8),$$
which simplifies to
$$8+ 16\sum_{i=3}^n (i - 28) = 8+8(n^2+n)-48-28(n-2) = 8n^2-20n+16.$$

\subsubsection{$\mathfrak{gl}(m|n)$}

We now proceed to the general case of $\mathfrak{gl}(m|n)$, where we assume that $m > n \geq 2$.  Note that in this case the ideal $\mathfrak{I}$ is defined slightly differently from how it is in the case where $m=n$, leading $\mathfrak{n}/\mathfrak{I}$ being isomorphic to the $\n$ from $\mathfrak{gl}(m-1|n)$.  Thus, using the spectral sequence decomposition
$$\operatorname{H}^2(\mathfrak{n}, \mathbb{C}) \cong \operatorname{H}^0(\mathfrak{n}/\mathfrak{I}, \Lambda_s^2(\mathfrak{I}^*)) \oplus \operatorname{H}^1(\mathfrak{n}/\mathfrak{I}, \mathfrak{I}^*) \oplus \operatorname{H}^2(\mathfrak{n}/\mathfrak{I}, \mathbb{C})$$
we can compute $\H(\n,\mathbb{C})$ recursively, working our way up from the $\n$ corresponding to $\mathfrak{gl}(n|n)$.

From here, the principals behind the computation are largely the same as in the $\mathfrak{gl}(n|n)$ case, where $\operatorname{H}^0(\mathfrak{n}/\mathfrak{I}, \Lambda_s^2(\mathfrak{I}^*))$ is computed by looking
at the fixed points of $\Lambda_s^2(\mathfrak{I}^*)$ and $\operatorname{H}^1(\mathfrak{n}/\mathfrak{I}, \mathfrak{I}^*)$ is computed by observing how the differentials act on weights.  Putting this all together, we obtain the following formulas for the dimension of $\H^2(\n, \mathbb{C})$ corresponding to $\mathfrak{gl}(n+\rho|n)$:

$$\dim{\H^2(\n, \mathbb{C})}=
\begin{cases}
8n^2-12n+8, & \rho=1 \\
8n^2-8n+8, & \rho=2 \\
8n^2-8n+8+4n(\rho-2)+\frac{(\rho-3)^2+(\rho-3)}{2}, & \rho >2. \\
\end{cases}$$
\subsubsection{$\mathfrak{q}(n)$}

The calculation of the dimension of the second cohomology for $\mathfrak{q}(n)$ is similar to that for $\mathfrak{gl}(n|n)$.  Note first that when $n = 2$, $\mathfrak{n}$ is a 2-dimensional, abelian Lie superalgebra, and so the $i$th cohomology will be isomorphic to $\Lambda^i_s(\mathfrak{n}^*)$.  Since both $\n^*_{\bar{0}}$ and $\n^*_{\bar{1}}$ have dimension 1, $\Lambda^i(\n^*_{\bar{0}}) = 0$ for all $i>0$ and $\operatorname{S}^j(\n^*_{\bar{1}})$ has dimension 1 for all $j$, so $\Lambda^i_s(\mathfrak{n}^*)$ is 2-dimensional for all $i$.
For general $n$, we may use the same direct sum decomposition derived from the spectral sequence as in  Equation~\ref{e:directsum}.

To compute $\operatorname{H}^0(\mathfrak{n}/\mathfrak{I}, \Lambda^2_s(I^*))$, which corresponds to fixed points of $\Lambda^2_s(I^*)$ under the action of $\mathfrak{n}/\mathfrak{I}$, note that again the only weight vectors of $\Lambda^2_s(I^*)$ that will vanish under the action of all elements $\mathfrak{n}/\mathfrak{I}$ will be superexterior products involving maximal even root and maximal odd root vectors, in particular, $\widetilde{E}_{n-1,n}^*$ and $\overbar{E}_{n-1,n}^*$.  Unlike in the $\mathfrak{gl}(n|n)$ case however, here there is only one such even root vector and one such odd root vector, so $\operatorname{H}^0(\mathfrak{n}/\mathfrak{I}, \Lambda^2_s(I^*))$ is spanned by $\widetilde{E}_{n-1,n}^* \otimes \overbar{E}_{n-1,n}^*$ and $\overbar{E}_{n-1,n}^* \otimes \overbar{E}_{n-1,n}^*$. Thus, in the case where $n>2$, the dimension of $\operatorname{H}^0(\mathfrak{n}/\mathfrak{I}, \Lambda^2_s(I^*))$ is equal to 2.

In computing the middle term $\operatorname{H}^1(\mathfrak{n}/\mathfrak{I}, \mathfrak{I}^*)$, we may again use the fact that we can decompose the terms of the cochain complex into their weight spaces, and the differentials will still preserve the action of the torus.  Again, we may look solely at weights from $(\n/[\n/\mathfrak{I},\n/\mathfrak{I}])^* \otimes I^*$ and argue as we did in the $\mathfrak{gl}(n|n)$ case.  Here, the kernel of $d^1$ will have dimension $4(n-2)+2(n-3)$ and the image of $d^0$ will have dimension $2(n-2)$, giving $\H^1(\n/\mathfrak{I}, \mathfrak{I}^*)$ a dimension of 4n-10.

Combining these terms together, we have the dimension of $\operatorname{H}^2(\mathfrak{n}, \mathbb{C})$ equals
$$2+ \sum_{i=3}^n(4i-8),$$
which simplifies to
$$2n^2-6n+6.$$

\subsubsection{$\mathfrak{osp}(2m|2n)$}

The same principles apply in computing the second cohomology for the $\mathfrak{osp}(2m|2n)$ superalgebras.  As before, we may decompose the cohomology into its direct sum decomposition as in Equation~\ref{e:directsum}.
Note that the last term is again computed recursively, starting with the base case $\mathfrak{osp}(2|2n)$. In this case, $\mathfrak{n}_{\bar{0}}$ is abelian, and so we obtain the direct sum decomposition:
$$H^2(\mathfrak{n},\mathbb{C}) \cong H^2(\mathfrak{n}_{\bar{0}},\mathbb{C}) \oplus H^1(\mathfrak{n}_{\bar{0}},\n_{\bar{1}}^*) \oplus H^2(\mathfrak{n}_{\bar{0}},\operatorname{S}^2(\n_{\bar{1}}^*)).$$
Since $\mathfrak{n}_{\bar{0}}$ is the nilpotent radical of an ordinary Lie algebra, these cohomologies may be computed via Kostant's theorem, which can be shown to sum up to have dimension $\frac{3n^2+n+4}{2}$. Using the fact for $\mathfrak{osp}(2m|2n)$, the $\mathfrak{n}/\mathfrak{I}$ is isomorphic to the $\mathfrak{n}$ from $\mathfrak{osp}(2(m-1)|n)$, the dimensions and weight space expressions for $\H^2(\mathfrak{n},\mathbb{C})$ for $m>1$ may then be computed recursively as in the case for $\mathfrak{gl}(m|n)$ and $\mathfrak{q}(n)$.  These are listed in the tables in Section 6.

\subsubsection{$\mathfrak{osp}(2m+1|2n)$}
We begin again with the direct sum decomposition from Equation~\ref{e:directsum}.
Much of the calculation is similar to that in the case of $\mathfrak{osp}(2m|2n)$.  We begin with the base case of $\mathfrak{osp}(3|2n)$ and use the recurrence from the direct sum formula to determine the weight space decomposition for any higher $\mathfrak{osp}(2m+1|2n).$

\subsubsection{$D(2,1,\alpha)$, $G(3)$, and $F(4)$}
Just as in the case of $\operatorname{H}^1(\mathfrak{n}, \mathbb{C})$, the second cohomology for $D(2, 1, \alpha)$, $G(3)$, and $F(4)$ can be easily computed using the fact that the corresponding subalgebras $\mathfrak{n}$ are abelian.  In particular, in each case $\operatorname{H}^2(\mathfrak{n}, \mathbb{C})$ is isomorphic to $C^2(\mathfrak{n}, \mathbb{C})$ in the corresponding cochain complex.  A description in terms of its weight space decomposition is given in the tables below.
\newpage

\section{Appendix: Tables of Weights and Dimensions}

In the tables below, we compile a list of all of the weights appearing in the first and second cohomologies for the Lie superalgebras used above, as well as their dimensions.  As a shorthand, we use the following notation.  For $\mathfrak{gl}(m|n)$, we let $\alpha_i$ be the weight $\epsilon_{i+1}-\epsilon_i$, $\alpha'_i$ be the weight $\delta_{i+1}-\delta_i$, $\beta_i$ be the weight $\delta_{i+1}-\epsilon_i$, and $\beta'_i$ be the weight $\epsilon_{i+1}-\delta_i$.  In the case of $\mathfrak{gl}(m|n)$, we assume that $m > n$.  In the case of $\mathfrak{osp}$, we let $\mu_1, \cdots \mu_m$ denote the simple weights of $B_m$ or $D_m$, and let $\nu_1, \cdots, \nu_n$ be the simple weights of $C_n$.

\subsection{$\operatorname{H}^1(\mathfrak{n}, \mathbb{C})$ Cohomology}

\noindent

\quad

\noindent

\quad

\begin{center}
\begin{tabular}{ |p{3.5cm}||p{3cm}|p{3cm}| }
 \hline
 \multicolumn{3}{|c|}{$\operatorname{H}^1(\mathfrak{n}, \mathbb{C})$ Cohomology (Classical Cases)} \\
 \hline
 \textbf{Lie Superalgebra} & \textbf{Corresponding Even Weights} & \textbf{Corresponding Odd Weights} \\
 \hline
  \hline
  $\mathfrak{gl}(n|n)$   & $\alpha_i$, $1 \leq i \leq n-1$, $\alpha'_j$, $1 \leq j \leq n-1$, & $\beta_i$, $1 \leq i \leq n-1$, $\beta'_j$, $1 \leq j \leq n-1$\\
  \hline
 $\mathfrak{gl}(m|n)$   & $\alpha_i$, $1 \leq i \leq m-1$, $\alpha'_j$, $1 \leq j \leq n-1$, & $\beta_i$, $1 \leq i \leq n-1$, $\beta'_j$, $1 \leq j \leq n$\\
 \hline
 $\mathfrak{osp}(2m|2n)$ &   $-\mu_i$, $1\leq i \leq m$
\linebreak
$-\nu_i$, $1\leq i \leq n$
&
$\epsilon_{i+1}-\delta_i$, $1\leq i \leq r$
\linebreak
$\delta_{i+1}-\epsilon_i$, $1\leq i \leq r$\\
 \hline
 $\mathfrak{osp}(2m+1|2n)$ &   $-\mu_i$, $1\leq i \leq m$
\linebreak
$-\nu_i$, $1\leq i \leq n$
&
$\epsilon_{i+1}-\delta_i$, $1\leq i \leq r$
\linebreak
$\delta_{i+1}-\epsilon_i$, $1\leq i \leq r$\\
 \hline
 $\mathfrak{q}(n)$ & $\epsilon_{i+1}-\epsilon_i$, \newline $1 \leq i \leq n-1$ & $\delta_{i+1}-\delta_i$, \newline $1 \leq i \leq n-1$ \\
 \hline
 \end{tabular}
\end{center}

\vfill

\begin{center}
\begin{tabular}{ |p{3.5cm}||p{3cm}|p{3.5cm}| }
 \hline
 \multicolumn{3}{|c|}{$\operatorname{H}^1(\mathfrak{n}, \mathbb{C})$ Cohomology (Exceptional Cases)} \\
 \hline
 \textbf{Lie Superalgebra} & \textbf{Corresponding Even Weights} & \textbf{Corresponding Odd Weights} \\
 \hline
  \hline
 $D(2,1,\alpha)$ & $-\mu_1$, $-\mu_2$, $-\mu_3$ & $(-\epsilon, -\epsilon, -\epsilon)$, $(-\epsilon, -\epsilon, \epsilon)$, $(\epsilon, -\epsilon, -\epsilon)$ \\
 \hline
 $G(3)$& $-\mu_1$, $-\alpha$, $-\beta$ & $(-\omega_1+\omega_2, -\epsilon)$,\newline $(2\omega_1-\omega_2, -\epsilon)$,\newline $(0, -\epsilon)$,\newline $(\omega_1-\omega_2, -\epsilon)$, \newline$(-2\omega_1+\omega_2, -\epsilon)$, \newline$(-\omega_1, -\epsilon)$ \\
 \hline
 $F(4)$ & $-\mu_1$, $-\nu_1$, $-\nu_2$, $-\nu_3$ &  $(\omega_2-\omega_3, -\epsilon)$,\newline $(\omega_1-\omega_2+\omega_3, -\epsilon)$,\newline $(\omega_1-\omega_3, -\epsilon)$, \newline$(-\omega_2+\omega_3, -\epsilon)$,\newline $(-\omega_1+\omega_2-\omega_3, -\epsilon)$,\newline $(-\omega_1+\omega_3, -\epsilon)$,\newline $(-\omega_3, -\epsilon)$ \\
 \hline
\end{tabular}
\end{center}

\pagebreak

\begin{center}
\begin{tabular}{ |p{3.5cm}||p{3cm}|p{3cm}|p{3cm}| }
 \hline
 \multicolumn{4}{|c|}{$\operatorname{H}^1(\mathfrak{n}, \mathbb{C})$ Cohomology Dimensions (Classical Cases)} \\
 \hline
 \textbf{Lie Superalgebra} & \textbf{Even} & \textbf{Odd} & \textbf{Total} \\
 \hline
  \hline
  $\mathfrak{gl}(n|n)$   & $2(n-1)$, & $2(n-1)$ & $4(n-1)$ \\
 \hline
 $\mathfrak{gl}(m|n)$   & $m-1+n-1$, & $2n-1$ & $m+3n-3$ \\
 \hline
 $\mathfrak{osp}(2m|2n)$ & $m+n-1$ & $m+n-1$ & $2m+2n-2$ \\
 \hline
 $\mathfrak{osp}(2m+1|2n)$ & $m+n$ & $2r$ & $m+n+2r$ \\
 \hline
 $\mathfrak{q}(n)$ & $n-1$ & $n-1$ & $2n-2$ \\
 \hline
 \end{tabular}
 
 \quad
 
 \quad
 
  \quad
 
 \quad
 
  \quad
 
 \quad
 
\begin{tabular}{ |p{3.5cm}||p{3cm}|p{3cm}|p{3cm}| }
 \hline
 
 \multicolumn{4}{|c|}{$\operatorname{H}^1(\mathfrak{n}, \mathbb{C})$ Cohomology Dimensions (Exceptional Cases)} \\
 \hline
  \hline
 \textbf{Lie Superalgebra} & \textbf{Even} & \textbf{Odd} & \textbf{Total} \\
 \hline
 $D(2,1,\alpha)$ & 3 & 3 & 6\\
 \hline
 $G(3)$& 3 & 6 & 9\\
 \hline
 $F(4)$ & 4 &  7 & 11 \\
 \hline
\end{tabular}
\end{center}

\pagebreak

\subsection{$\operatorname{H}^2(\mathfrak{n}, \mathbb{C})$ Cohomology}

Note that every weight in the $\operatorname{H}^2(\mathfrak{n},\mathbb{C})$ corresponds to the sum of two roots of the Lie superalgebra.  Below, we classify the weights by whether they are the sum of two even roots, two odd roots, or of an even root and and odd root.

\vfill

\begin{center}
\begin{tabular}{ |p{3.5cm}||p{4cm}|p{4cm}|p{4cm}| }
 \hline
 \multicolumn{4}{|c|}{$\operatorname{H}^2(\mathfrak{n}, \mathbb{C})$ Cohomology (Classical Cases)} \\
 \hline
 \textbf{Lie Superalgebra} & \textbf{Even+Even Weights} & \textbf{Odd+Odd Weights} & \textbf{Odd+Even Weights} \\
 \hline
  \hline
   $\mathfrak{gl}(n|n)$   & $\alpha_i+\alpha_j$,\newline $1 \leq i<j \leq n-1$, \newline $\alpha'_i+\alpha'_j$, \newline $1 \leq i<j \leq n-1$, \newline 
   $\alpha_i + \beta_j, \, \alpha_i+\beta_j'$, \newline $1 \leq i <i+1<j \leq n-1$
 \newline
 $\alpha'_i + \beta_j, \, \alpha'+\beta_j'$, \newline $1 \leq i <i+1<j \leq n-1$
 
 & $\alpha_i \pm \beta_j$, \newline $1 \leq i <i+1<j \leq r-1$
 \newline
 $\alpha'_i \pm \beta_j$, \newline $1 \leq i <i+1<j \leq r-1$
 \newline &  $\pm \beta_i \pm \beta_j$, \newline $1 \leq i <j \leq r$
 \\
 \hline
  
 $\mathfrak{gl}(m|n)$   & $\alpha_i+\alpha_j$,\newline $1 \leq i<i+1<j \leq m-1$, \newline $\alpha'_i+\alpha'_j$, \newline $1 \leq i<i+1<j \leq n-1$, \newline 
 
 & $\alpha_i \pm \beta_j$, \newline $1 \leq i <i+1<j \leq r-1$
 \newline
 $\alpha'_i \pm \beta_j$, \newline $1 \leq i <i+1<j \leq r-1$
 \newline &  $\pm \beta_i \pm \beta_j$, \newline $1 \leq i <j \leq r$
 \\
 \hline
 
  $\mathfrak{osp}(2m|2n)$ &   $\mu_i+\mu_j$, \newline $1 \leq i<i+1<j \leq m-1$, \newline $\nu'_i+\nu'_j$,\newline $1 \leq i<i+1<j \leq n-1$, \newline\ 

  &  $\mu_i \pm \delta_{j+1}-\epsilon_j$, \newline$1 \leq i <i+1<j \leq r-1$ \newline
  $\nu_i \pm \delta_{j+1}-\epsilon_j$, \newline $1 \leq i <i+1<j \leq r-1$ &  $\pm (\delta_{i+1}-\epsilon_i) \pm (\delta_{j+1}-\epsilon_j)$,  \newline $1 \leq i <j \leq r$ \\
 \hline
 
 $\mathfrak{osp}(2m+1|2n)$ & $\mu_i+\mu_j$,\newline $1 \leq i<i+1<j \leq m-1$, \newline $\nu'_i+\nu'_j$, \newline $1 \leq i<i+1<j \leq n-1$, \newline

   & $\mu_i \pm \delta_{j+1}-\epsilon_j$, \newline $1 \leq i <i+1<j \leq r-1$
 \newline $\nu_i \pm \delta_{j+1}-\epsilon_j$, \newline $1 \leq i <i+1<j \leq r-1$ & $\pm (\delta_{i+1}-\epsilon_i) \pm (\delta_{j+1}-\epsilon_j)$, \newline $1 \leq i <j \leq r$\\
 \hline
 $\mathfrak{q}(n)$ & $\alpha_i+\alpha_j$, \newline$1 \leq i<i+1<j \leq n-1$,

 &  $\alpha_i \pm \beta_j$, \newline $1 \leq i <i+1<j \leq n-1$,
 & $\pm \beta_i \pm \beta_j$,\newline  $1 \leq i <j \leq n-1$\\
 \hline
 \end{tabular}
\end{center}

\vfill

 \begin{center}
\begin{tabular}{ |p{3.5cm}||p{4cm}|p{4cm}|p{4cm}| }
 \hline
 \multicolumn{4}{|c|}{$\operatorname{H}^2(\mathfrak{n}, \mathbb{C})$ Cohomology (Exceptional Cases)} \\
 \hline
  \textbf{Lie Superalgebra} & \textbf{Even+Even Weights} & \textbf{Odd+Odd Weights} & \textbf{Odd+Even Weights} \\
 \hline
  \hline
 $D(2,1,\alpha)$ & Sums of any distinct two following weights: \linebreak $-\mu_1$, $-\mu_2$, $-\mu_3$ & Sums of one weight from left column with one from right column &Sums of any two of the following weights: $(-\epsilon, -\epsilon, -\epsilon)$, \newline$(-\epsilon, -\epsilon, \epsilon)$,\newline $(\epsilon, -\epsilon, -\epsilon)$\\
 \hline
 $G(3)$ & Sums of any distinct two following weights: \linebreak$-\mu_1$, $-\alpha$, $-\beta$& Sums of one weight from left column with one from right column &Sums of any two of the following weights: \linebreak$(-\omega_1+\omega_2, -\epsilon)$,\newline $(2\omega_1-\omega_2, -\epsilon)$, \newline$(0, -\epsilon)$, \newline$(\omega_1-\omega_2, -\epsilon)$,\newline $(-2\omega_1+\omega_2, -\epsilon)$,\newline $(-\omega_1, -\epsilon)$\\
 \hline
 $F(4)$ & Sums of any distinct two following weights: \linebreak$-\mu_1$, $-\nu_1$, $-\nu_2$, $-\nu_3$ & Sums of one weight from left column with one from right column &Sums of any two of the following weights: \linebreak $(\omega_2-\omega_3, -\epsilon)$, \newline $(\omega_1-\omega_2+\omega_3, -\epsilon)$,\newline $(\omega_1-\omega_3, -\epsilon)$, \newline$(-\omega_2+\omega_3, -\epsilon)$,\newline $(-\omega_1+\omega_2-\omega_3, -\epsilon)$,\newline $(-\omega_1+\omega_3, -\epsilon)$, \newline$(-\omega_3, -\epsilon)$\\
 \hline
\end{tabular}
\end{center}

\pagebreak

\begin{center}
\begin{tabular}{ |p{3.5cm}||p{3.1cm}|p{3cm}|p{3cm}|p{3cm}| }
 \hline
 \multicolumn{5}{|c|}{$\operatorname{H}^2(\mathfrak{n}, \mathbb{C})$ Cohomology Dimensions (Classical Cases)} \\
 \hline
 \textbf{Lie Superalgebra} & \textbf{Even+Even} & \textbf{Odd+Odd} & \textbf{Odd+Even} & \textbf{Total}\\
  \hline
 \hline
 $\mathfrak{gl}(n|n)$   & $4n^2-10n+8$ 
 
 & $4n^2-10n+8$ &
 $2(r^2+r)$ & $8n^2-20n+16$
 \\
 \hline
 $\mathfrak{gl}(m|n)$   & $\frac{1}{2}((n-1)^2+(n-1) \linebreak +(m-1)^2+(m-1)$ 
 
 & $(n+m-2)(2r)$&
 $2(r^2+r)$ & $\frac{1}{2}((n-1)^2+(n-1) +(m-1)^2+(m-1)+ (n+m-2)(2r)+2(r^2+r)$
 \\
 \hline
 
  $\mathfrak{osp}(2m|2n)$ & $\frac{1}{2}((n-1)^2+(n-1) \linebreak +(m-1)^2+(m-1)$ 
 
 & $(n+m-2)(2r)$&
 $2(r^2+r)$ & $\frac{1}{2}((n-1)^2+(n-1)+(m-1)^2+(m-1)+ (n+m-2)(2r)+2(r^2+r)$
 \\
 \hline
 
 $\mathfrak{osp}(2m+1|2n)$ & $\frac{1}{2}((n-1)^2+(n-1) \linebreak +(m-1)^2+(m-1)$ 
 
 & $(n+m-2)(2r)$&
 $2(r^2+r)$ & $\frac{1}{2}((n-1)^2+(n-1)+(m-1)^2+(m-1)+ (n+m-2)(2r)+2(r^2+r)$
 \\
 \hline
 
 $\mathfrak{q}(n)$ & $\frac{1}{2}((n-1)^2+(n-1))$ & $(n-1)^2$ & $\frac{1}{2}((n-1)^2+(n-1)^2$ & $2(n-1)^2+(n-1)$ \\
 \hline
 \end{tabular}
 \end{center}

 \begin{center}
\begin{tabular}{ |p{3.5cm}||p{3cm}|p{3cm}|p{3cm}|p{3cm}| }
 \hline
 \multicolumn{5}{|c|}{$\operatorname{H}^2(\mathfrak{n}, \mathbb{C})$ Cohomology Dimensions (Exceptional Cases)} \\
 \hline
 \textbf{Lie Superalgebra} & \textbf{Even+Even} & \textbf{Odd+Odd} & \textbf{Odd+Even} & \textbf{Total}\\
 \hline
  \hline
 $D(2,1,\alpha)$ & 3 & 9 & 6 & 18\\
 \hline
 $G(3)$ & 3 & 18 & 21 & 42 \\
 \hline
 $F(4)$ & 6 & 28 & 28 & 62 \\
 \hline
\end{tabular}
\end{center}

\pagebreak

\bibliographystyle{amsalpha}
\bibliography{bibliography}

\end{document}